\documentclass[a4paper,12pt]{article}
\usepackage{graphicx}
\usepackage{amsfonts}
\usepackage{amsthm}
\usepackage{pgfplots}
\usepackage{amsmath}
\usepackage{float}
\usepackage{enumerate}
\usepackage{tikz}
\usetikzlibrary{positioning,decorations.pathmorphing,arrows.meta}

\title{\textbf{Monochromatic Paths in the Complete Symmetric Infinite Digraph}}
\author{Hannah Guggiari\thanks{Mathematical Institute, University of Oxford, Oxford, OX2 6GG, United Kingdom.}}
\date{\today}

\tikzset{main node/.style={circle,fill=black,draw,minimum width=4pt,inner sep=0pt}}

\newtheorem{theorem}{Theorem}[section]
\newtheorem{conjecture}[theorem]{Conjecture}

\begin{document}
\newpage
\maketitle
\begin{abstract}
\noindent
Let $\vec{K}_{\mathbb{N}}$ be the complete symmetric digraph on the positive integers. Answering a question of DeBiasio and McKenney \cite{DM16}, we construct a 2-colouring of the edges of $\vec{K}_{\mathbb{N}}$ in which every monochromatic path has density 0. On the other hand, we show that, in every colouring that does not have a directed path with $r$ edges in the first colour, there is directed path in the second colour with density at least $\frac1r$.
\end{abstract}
\noindent
\section{Introduction}
Let $K_{\mathbb{N}}$ be the complete graph on the positive integers and $\vec{K}_{\mathbb{N}}$ be the complete symmetric digraph on the positive integers. The \textit{upper density} of a set $A\subseteq\mathbb{N}$ is
\[
\bar{d}(A)=\limsup_{n\rightarrow\infty}\frac{|A\cap\{1,\dots,n\}|}n
\]
For a graph or digraph $G$ with vertex set $V(G)\subseteq\mathbb{N}$, we define the \textit{upper density} of $G$ to be that of $V(G)$.  Throughout this paper, by a $k$-colouring, we mean a $k$-edge-colouring. In a 2-colouring, we will assume that the colours are red and blue.

For undirected graphs, Rado \cite{R78} showed that, in every 2-colouring of $K_{\mathbb{N}}$, there is a partition of the vertices into 2 disjoint paths of distinct colours. Elekes, Soukup, Soukup and Szentmikl\`{o}ssy \cite{ESSS17} have recently extended this result to the complete graph on $\omega_1$ where $\omega_1$ is the smallest uncountable cardinal. Erd\H{o}s and Galvin \cite{EG93} showed that if the edges of $K_{\mathbb{N}}$ are coloured with 2 colours, then there exists a monochromatic infinite path with upper density at least $\frac23$.

For directed graphs, the picture is a little different. In the finite case, Raynaud \cite{R73} showed that, in any 2-colouring of $\vec{K}_n$, there is a directed Hamiltonian cycle $C$ with the following property: there are two vertices $a$ and $b$ such that the directed path from $a$ to $b$ along $C$ is red and the directed path from $b$ to $a$ along $C$ is blue.

In this paper, we will be interested in the infinite directed case. In particular, we will be considering edge-colourings of $\vec{K}_{\mathbb{N}}$ and will prove a variety of results relating to the upper density of paths in $\vec{K}_{\mathbb{N}}$.

Let $P=v_1v_2...$ be a path in $\vec{K}_{\mathbb{N}}$. We say that $P$ is a \textit{directed path} if every edge in $P$ is oriented in the same direction. By the \textit{length} of a path, we mean the number of edges in the path. DeBiasio and McKenney \cite{DM16} recently proved the following result.

\begin{theorem}
\label{thm:densityepsilon}
For every $\varepsilon>0$, there exists a 2-colouring of $\vec{K}_{\mathbb{N}}$ such that every monochromatic directed path has upper density less than $\varepsilon$.
\end{theorem}

\noindent
DeBiasio and McKenney also asked the following natural question: does there exist a 2-colouring of $\vec{K}_{\mathbb{N}}$ in which every monochromatic directed path has upper density 0? In Section \ref{density0}, we will give a positive answer to this question.

\begin{theorem}
\label{thm:density0}
There exists a 2-colouring of $\vec{K}_{\mathbb{N}}$ such that every monochromatic directed path has upper density 0.
\end{theorem}
\noindent
In light of this result, it is natural to ask under what conditions we can guarantee the existence of a monochromatic path of positive density. It is easy to see (from Ramsey's Theorem) that every 2-colouring of $\vec{K}_{\mathbb{N}}$ contains a monochromatic directed path of infinite length. In Section \ref{restrictred}, we will consider what happens if we restrict the length of directed paths in one of the colours. In particular, we will prove the following density result.

\begin{theorem}
\label{thm:restrictedred}
Take any 2-colouring of $\vec{K}_{\mathbb{N}}$ that does not contain a red directed path of length $r$. Then there exists a blue directed path with upper density at least $\frac1r$.
\end{theorem}
\noindent
This result is the best possible as shown by the following construction. Divide the vertices of $\vec{K}_{\mathbb{N}}$ into sets $A_1,\dots,A_r$ based on their residue modulo $r$. So, for any $n\in\mathbb{N}$, we have $n\in A_i$ if and only if $n\equiv i\mod r$. Consider distinct vertices $m,n\in\mathbb{N}$. If $m,n\in A_i$, then colour both $(m,n)$ and $(n,m)$ blue. If $m\in A_i$, $n\in A_j$ and $i<j$, then colour $(m,n)$ blue and $(n,m)$ red (see Figure \ref{fig:extremalred}). We call this colouring $c_r$. It should be noted that each of the sets $A_i$ has density $\frac1r$.

\begin{figure}[H]
\centering
\begin{tikzpicture}
\draw (-6,0) ellipse (1cm and 2cm) node {$A_0$};
\draw (-2,0) ellipse (1cm and 2cm) node {$A_1$};
\draw (2,0) ellipse (1cm and 2cm) node {$A_2$};
\draw (6,0) ellipse (1cm and 2cm) node {$A_3$};
\draw[-{Latex[length=2mm,width=2mm]},red,thick] (5.5,0.5) to (2.5,0.5);
\draw[-{Latex[length=2mm,width=2mm]},red,thick] (1.5,0.5) to (-1.5,0.5);
\draw[-{Latex[length=2mm,width=2mm]},red,thick] (-2.5,0.5) to (-5.5,0.5);
\draw[-{Latex[length=2mm,width=2mm]},blue,thick] (-5.5,-0.5) to (-2.5,-0.5);
\draw[-{Latex[length=2mm,width=2mm]},blue,thick] (-1.5,-0.5) to (1.5,-0.5);
\draw[-{Latex[length=2mm,width=2mm]},blue,thick] (2.5,-0.5) to (5.5,-0.5);
\draw[-{Latex[length=2mm,width=2mm]},red,thick] (5.5,1) to [bend right] (-1.5,1);
\draw[-{Latex[length=2mm,width=2mm]},red,thick] (1.5,1) to [bend right] (-5.5,1);
\draw[-{Latex[length=2mm,width=2mm]},red,thick] (6,1) to [bend right] (-6,1);
\draw[-{Latex[length=2mm,width=2mm]},blue,thick] (-5.5,-1) to [bend right] (1.5,-1);
\draw[-{Latex[length=2mm,width=2mm]},blue,thick] (-1.5,-1) to [bend right] (5.5,-1);
\draw[-{Latex[length=2mm,width=2mm]},blue,thick] (-6,-1) to [bend right] (6,-1);
\draw[-{Latex[length=2mm,width=2mm]},blue,thick] (-6.3,0.5) to (-6.3,-0.5);
\draw[-{Latex[length=2mm,width=2mm]},blue,thick] (-6.5,-0.5) to (-6.5,0.5);
\draw[-{Latex[length=2mm,width=2mm]},blue,thick] (-1.7,0.5) to (-1.7,-0.5);
\draw[-{Latex[length=2mm,width=2mm]},blue,thick] (-2.3,-0.5) to (-2.3,0.5);
\draw[-{Latex[length=2mm,width=2mm]},blue,thick] (1.7,0.5) to (1.7,-0.5);
\draw[-{Latex[length=2mm,width=2mm]},blue,thick] (2.3,-0.5) to (2.3,0.5);
\draw[-{Latex[length=2mm,width=2mm]},blue,thick] (6.3,0.5) to (6.3,-0.5);
\draw[-{Latex[length=2mm,width=2mm]},blue,thick] (6.5,-0.5) to (6.5,0.5);
\end{tikzpicture}
\caption{Extremal colouring $c_4$.}
\label{fig:extremalred}
\end{figure}
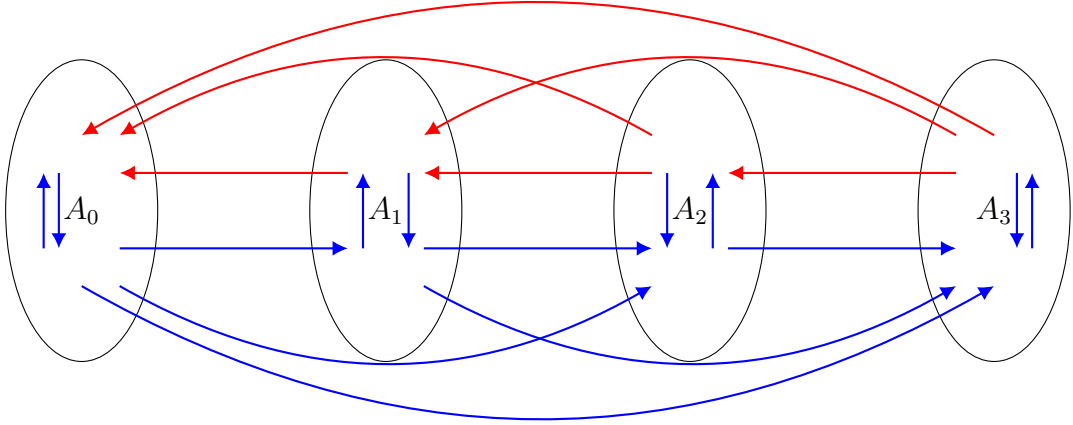

\noindent
It is straightforward to see that the longest red directed path has length $r-1$. Any blue directed path has infinite intersection with at most one of the sets $A_i$ and therefore has upper density at most $\frac1r$.

In Section \ref{stability}, we will consider the stability of this result. We will prove the following theorem.

\begin{theorem}
\label{thm:stability}
Take any 2-colouring of $\vec{K}_{\mathbb{N}}$ in which there are no red directed paths of length $r$ and every blue directed path has upper density at most $\frac1r$. Then, there exists a finite set of vertices $U$ such that the 2-colouring induced on $\mathbb{N}\backslash U$ is isomorphic to the 2-colouring $c_r$ described above.
\end{theorem}
\noindent
Finally, in Section \ref{restrictedmanycolours}, we will extend the result of Theorem \ref{thm:restrictedred} to include $(k+1)$-colourings of $\vec{K}_{\mathbb{N}}$ for any $k\geq2$.

\begin{theorem}
\label{thm:restrictedmanycolours}
Fix $k\geq2$ and consider any $(k+1)$-colouring of $\vec{K}_{\mathbb{N}}$ using colours $c_1,\dots,c_{k+1}$. Suppose that, for every $t\in[k]$, there is no $c_t$-coloured directed path of length $r_t$. Then there exists a directed path of colour $c_{k+1}$ with upper density at least $\prod_{t=1}^k\frac1{r_t}$.
\end{theorem}

\section{All Monochromatic Paths Have Density 0}
\label{density0}
\begin{proof}[Proof of Theorem \ref{thm:density0}]
We colour the edges of $\vec{K}_{\mathbb{N}}$ in the following way. Let $m,n\in\mathbb{N}$ be distinct positive integers. Set $t=\min\{s\in\mathbb{N}:m\not\equiv n\mod2^s\}$. Exchanging $m$ and $n$ if necessary, we may assume that $m\equiv x\mod2^t$ where $x\in\{0,\dots,2^{t-1}-1\}$ and $n\equiv2^{t-1}+x\mod2^t$. We colour $(m,n)$ red and $(n,m)$ blue.

Let $P$ be any monochromatic directed path. If $P$ is a finite path, then $\bar{d}(P)=0$. Therefore, we may assume that $P$ is an infinite path. Without loss of generality, $P$ is red.
\\
\\
\textit{Inductive Hypothesis:} For any $k\in\mathbb{N}$, there exists $i\in\{0,\dots,2^k-1\}$ such that $P$ is eventually contained within the set $\{n\in\mathbb{N}:n\equiv i\mod2^k\}$. Hence, $\bar{d}(P)\leq2^{-k}$.
\\
\\
\textit{Base Case:} For $i\in\{0,1\}$, let $A_i=\{n\in\mathbb{N}:n\equiv i\mod2\}$. The sets $A_0$ and $A_1$ partition the vertices of $\vec{K}_{\mathbb{N}}$. Suppose $P$ contains a vertex $u\in A_1$. Then all of the vertices occurring after $u$ in $P$ must also be in $A_1$ because $P$ is a red directed path and hence $P$ is eventually contained within $A_1$. If $P$ does not contain a vertex from $A_1$, then $P$ must be completely contained within $A_0$. Hence, for some $i\in\{0,1\}$, we have $\bar{d}(P)\leq\bar{d}(A_i)=\frac12$.
\\
\\
\textit{Inductive Step:} Fix $k\geq2$ and partition the vertices of $\vec{K}_{\mathbb{N}}$ into $2^k$ sets $A_0,\dots,A_{2^k-1}$ based on their residue modulo $2^k$ (see Figure \ref{fig:densityzero}). By the inductive hypothesis, there exists $i\in\{0,\dots,2^{k-1}-1\}$ such that $P$ is eventually contained within the set $A_i\cup A_{2^{k-1}+i}$. By using the same argument as in the base case, we find that, if $P$ contains a vertex in $A_{2^{k-1}+i}$, then it is eventually contained within this set; otherwise, it is eventually contained within $A_i$. Hence, there exists $j\in\{0,\dots,2^k-1\}$ such that $P$ is eventually contained within $A_j$ and so $\bar{d}(P)\leq\bar{d}(A_j)=2^{-k}$.

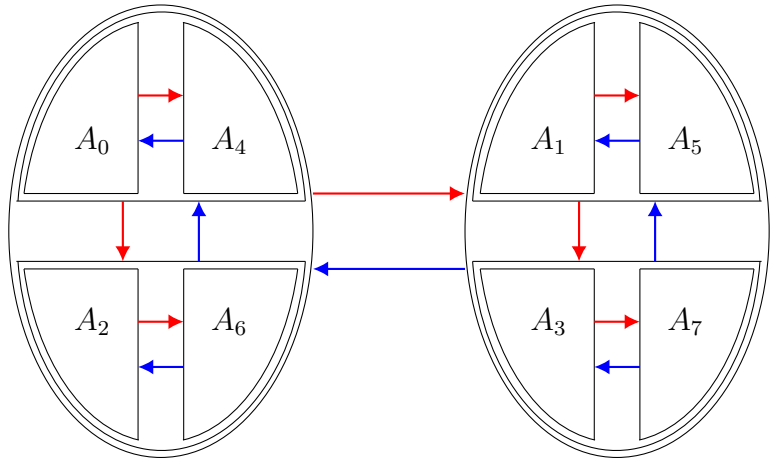
\begin{figure}[H]
\centering
\begin{tikzpicture}
\draw (-3,0) ellipse (2cm and 3cm) node {};
\draw (3,0) ellipse (2cm and 3cm) node {};
\draw[-{Latex[length=2mm,width=2mm]},blue,thick] (1,-0.5) to (-1,-0.5);
\draw[-{Latex[length=2mm,width=2mm]},red,thick] (-1,0.5) to (1,0.5);
\draw (-1.1,0.4) arc (6:174:1.9cm and 2.8cm);
\draw (-1.1,0.4) -- (-4.9,0.4);
\draw (-1.1,-0.4) arc (6:174:1.9cm and -2.8cm);
\draw (-1.1,-0.4) -- (-4.9,-0.4);
\draw[-{Latex[length=2mm,width=2mm]},blue,thick] (-2.5,-0.4) to (-2.5,0.4);
\draw[-{Latex[length=2mm,width=2mm]},red,thick] (-3.5,0.4) to (-3.5,-0.4);
\draw (1.1,0.4) arc (6:174:-1.9cm and 2.8cm);
\draw (1.1,0.4) -- (4.9,0.4);
\draw (1.1,-0.4) arc (6:174:-1.9cm and -2.8cm);
\draw (1.1,-0.4) -- (4.9,-0.4);
\draw[-{Latex[length=2mm,width=2mm]},blue,thick] (3.5,-0.4) to (3.5,0.4);
\draw[-{Latex[length=2mm,width=2mm]},red,thick] (2.5,0.4) to (2.5,-0.4);
\draw (1.2,0.5) arc (12:82:-1.8cm and 2.9cm);
\draw (1.2,0.5) -- (2.7,0.5);
\draw (2.7,2.77) -- (2.7,0.5);
\draw (4.8,0.5) arc (12:82:1.8cm and 2.9cm);
\draw (4.8,0.5) -- (3.3,0.5);
\draw (3.3,2.77) -- (3.3,0.5);
\draw[-{Latex[length=2mm,width=2mm]},blue,thick] (3.3,1.2) to (2.7,1.2);
\draw[-{Latex[length=2mm,width=2mm]},red,thick] (2.7,1.8) to (3.3,1.8);
\draw (-1.2,0.5) arc (12:82:1.8cm and 2.9cm);
\draw (-1.2,0.5) -- (-2.7,0.5);
\draw (-2.7,2.77) -- (-2.7,0.5);
\draw (-4.8,0.5) arc (12:82:-1.8cm and 2.9cm);
\draw (-4.8,0.5) -- (-3.3,0.5);
\draw (-3.3,2.77) -- (-3.3,0.5);
\draw[-{Latex[length=2mm,width=2mm]},blue,thick] (-2.7,1.2) to (-3.3,1.2);
\draw[-{Latex[length=2mm,width=2mm]},red,thick] (-3.3,1.8) to (-2.7,1.8);
\draw (1.2,-0.5) arc (12:82:-1.8cm and -2.9cm);
\draw (1.2,-0.5) -- (2.7,-0.5);
\draw (2.7,-2.77) -- (2.7,-0.5);
\draw (4.8,-0.5) arc (12:82:1.8cm and -2.9cm);
\draw (4.8,-0.5) -- (3.3,-0.5);
\draw (3.3,-2.77) -- (3.3,-0.5);
\draw[-{Latex[length=2mm,width=2mm]},blue,thick] (3.3,-1.8) to (2.7,-1.8);
\draw[-{Latex[length=2mm,width=2mm]},red,thick] (2.7,-1.2) to (3.3,-1.2);
\draw (-1.2,-0.5) arc (12:82:1.8cm and -2.9cm);
\draw (-1.2,-0.5) -- (-2.7,-0.5);
\draw (-2.7,-2.77) -- (-2.7,-0.5);
\draw (-4.8,-0.5) arc (12:82:-1.8cm and -2.9cm);
\draw (-4.8,-0.5) -- (-3.3,-0.5);
\draw (-3.3,-2.77) -- (-3.3,-0.5);
\draw[-{Latex[length=2mm,width=2mm]},blue,thick] (-2.7,-1.8) to (-3.3,-1.8);
\draw[-{Latex[length=2mm,width=2mm]},red,thick] (-3.3,-1.2) to (-2.7,-1.2);
\node at (-3.9,1.2) {$A_0$};
\node at (2.1,1.2) {$A_1$};
\node at (-3.9,-1.2) {$A_2$};
\node at (2.1,-1.2) {$A_3$};
\node at (-2.1,1.2) {$A_4$};
\node at (3.9,1.2) {$A_5$};
\node at (-2.1,-1.2) {$A_6$};
\node at (3.9,-1.2) {$A_7$};
\end{tikzpicture}
\caption{Diagram showing the edges between sets for $k=3$.}
\label{fig:densityzero}
\end{figure}
\noindent
The inductive hypothesis holds for every $k\in\mathbb{N}$. Therefore, if $P$ is any monochromatic directed path, we have that the upper density of $P$ is at most $2^{-k}$ for every $k\in\mathbb{N}$. Hence, $P$ has upper density 0.
\end{proof}

\section{Restricting Red Path Length}
\label{restrictred}
\begin{proof}[Proof of Theorem \ref{thm:restrictedred}]
Fix a 2-colouring of $\vec{K}_{\mathbb{N}}$ and assume that there is no red directed path of length $r$. Partition the vertices into sets $A_0,\dots,A_{r-1}$ where $n\in A_i$ if and only if the longest directed red path from $n$ has length $i$.
\\
\\
\textit{Claim 1: Let $v\in A_i$. Then the red indegree of $v$ in $A_i$ is at most $i$.}
\\
\noindent
Take a directed red path of maximal length $vv_{i-1}\dots v_0$. Let $U=\{u\in A_i: (u,v) \text{ is red}\}$. For every $u\in U$, there must exist $t\in\{0,\dots,i-1\}$ such that $u=v_t$ or else $uvv_{i-1}\dots v_0$ would be a directed red path of length $i+1$ from $u$, which contradicts $u\in A_i$.  Hence $|U|\leq i$.
\\
\\
At least one of the sets $A_i$ must have upper density at least $\frac1r$. Choose such an $A_i$ and let $B=\{v\in A_i:v\text{ has finite blue outdegree in }A_i\}$.
\\
\\
\textit{Claim 2: $|B|\leq i$.}
\\
\noindent
Suppose not and consider a subset $\{b_1,\dots,b_{i+1}\}\subseteq B$ of size $i+1$. We may list the vertices of $A_i$ in order $v_1,v_2,\dots$. For each $1\leq j\leq i+1$, define $t_j=\min\{k\in\mathbb{N}:\forall k'\geq k\text{ the edge }(b_j,v_{k'})\text{ is red}\}$. This is well-defined because $b_j$ has finite blue outdegree. Taking $t=\max\{t_1,\dots,t_{i+1}\}$, we find that the vertex $v_t$ has red indegree at least $i+1$ in $A_i$, which contradicts the result of Claim 1.
\\
\\
As $A_i$ is infinite and $B$ is finite, $A_i\backslash B$ is also infinite and $\bar{d}(A_i)=\bar{d}(A_i\backslash B)$. We may therefore assume, without loss of generality, that $B=\emptyset$. Every vertex in $A_i$ has infinite blue outdegree in $A_i$. Listing the vertices of $A_i$ in order $v_1,v_2,\dots$, we create a blue directed path $P=\{p_j\}$ as follows:
\begin{itemize}
\item Set $p_1=v_1$.
\item Given $p_j$, let $t=\min\{k:v_k\notin P\}$. If the edge $(p_j,v_t)$ is blue, set $p_{j+1}=v_t$. If not, there exists $u\notin P$ such that the edges $(p_j,u)$ and $(u,v_t)$ are both blue. Such a vertex $u$ always exists because $p_j$ has an infinite number of blue out-neighbours and $v_t$ has at most $i$ red in-neighbours. Set $p_{j+1}=u$ and $p_{j+2}=v_t$.
\end{itemize}
The blue directed path $P$ contains every vertex of $A_i$ and hence has upper density $\bar{d}(A_i)$.
\end{proof}

\section{Stability Result}
\label{stability}
\begin{proof}[Proof of Theorem \ref{thm:stability}]
Fix a 2-colouring of $\vec{K}_{\mathbb{N}}$ such that there is no red directed path of length $r$ and every blue directed path has density at most $\frac1r$. Partition the vertices into sets $A_0,\dots,A_{r-1}$ where $v\in A_i$ if and only if the longest directed red path from $v$ has length $i$.

By using the same argument as in the proof of Theorem \ref{thm:restrictedred}, we also know that, if $A_i$ is an infinite set, then there exists a blue directed path within $A_i$ of density $\bar{d}(A_i)$. As there are $r$ sets $A_i$, we must have $\bar{d}(A_i)=\frac1r$ for every $i\in\{0,\dots,r-1\}$.

We will prove by induction that, by removing a finite number of vertices, all edges from $A_{i-1}$ to $A_i$ are blue and all edges from $A_i$ to $A_{i-1}$ are red for every $i\in\{1,\dots,r-1\}$ (see Figure \ref{fig:induction}).

\begin{figure}[H]
\centering
\begin{tikzpicture}
\draw (-4.5,0) ellipse (1cm and 1.5cm) node {$A_0$};
\draw (-1.5,0) ellipse (1cm and 1.5cm) node {$A_1$};
\draw (1.5,0) ellipse (1cm and 1.5cm) node {$A_2$};
\draw (4.5,0) ellipse (1cm and 1.5cm) node {$A_3$};
\draw[-{Latex[length=2mm,width=2mm]},red,thick] (4,0.5) to (2,0.5);
\draw[-{Latex[length=2mm,width=2mm]},red,thick] (1,0.5) to (-1,0.5);
\draw[-{Latex[length=2mm,width=2mm]},red,thick] (-2,0.5) to (-4,0.5);
\draw[-{Latex[length=2mm,width=2mm]},blue,thick] (-4,-0.5) to (-2,-0.5);
\draw[-{Latex[length=2mm,width=2mm]},blue,thick] (-1,-0.5) to (1,-0.5);
\draw[-{Latex[length=2mm,width=2mm]},blue,thick] (2,-0.5) to (4,-0.5);
\end{tikzpicture}
\caption{Result of the induction for $r=4$.}
\label{fig:induction}
\end{figure}
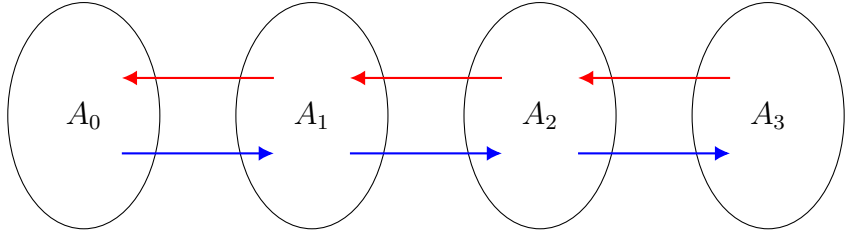

\noindent
\textit{Induction Hypothesis:} For fixed $k\geq1$, by removing a finite number of vertices, $A_0\cup\dots\cup A_{k}$ has the following structure: for every $i\in\{1,\dots,k\}$, all edges from  $A_{i-1}$ to $A_i$ are blue and all edges from $A_i$ to $A_{i-1}$ are red.
\\
\\
\textit{Base Case:} Let $k=1$. By definition of $A_0$, all edges from $A_0$ to $A_1$ are blue. If there are no blue edges from $A_1$ to $A_0$, we are done. Otherwise, we form a matching $M$ of blue edges from $A_1$ to $A_0$ by adding blue edges one by one until we cannot add any more.

Suppose $M$ is a finite matching. Let $B$ be the set of vertices which are incident to an edge in $M$. By definition, every blue edge from $A_1$ to $A_0$ has at least one endpoint in $B$. Therefore, by removing the vertices in $B$, we will ensure that every edge from $A_1$ to $A_0$ is red as required.

Suppose instead that $M$ is an infinite matching. We noted earlier that $A_0$ contains a blue directed path $P=\{p_j\}$ and $A_1$ contains a blue directed path $Q=\{q_k\}$, both of which have density $\frac1r$. We create a blue directed path $S$ as follows:
\begin{itemize}
\item Pick an edge $(q_k,p_j)\in M$ where $j>1$. Set $S=p_1\dots p_{j-1}q_1\dots q_kp_j$. The edge $(p_{j-1},q_1)$ is blue because all edges from $A_0$ to $A_1$ are blue.
\item Suppose the final two vertices of $S$ are $q_{k_1}p_{j_1}$ where $(q_{k_1},p_{j_1})\in M$. Pick an edge $(q_{k_2},p_{j_2})\in M$ such that $k_1<k_2$ and $j_1<j_2$. Such an edge exists because $M$ is an infinite matching and $S$ only contains a finite number of vertices. We note that the edge $(p_{j_2-1},q_{k_1+1})$ is blue because it is from $A_0$ to $A_1$. Add the subpath $p_{j_1+1}\dots p_{j_2-1}q_{k_1+1}\dots q_{k_2}p_{j_2}$ to the end of $S$ .
\end{itemize}
\noindent
By construction, the blue directed path $S$ uses every vertex in both $P$ and $Q$ and so $\bar{d}(S)=\bar{d}(P)+\bar{d}(Q)=\frac2r$.

We assumed that every blue path had density at most $\frac1r$ so this is a contradiction. Therefore, any blue matching from $A_1$ to $A_0$ is finite and so we are done.
\\
\\
\textit{Inductive Step:} Now let $k>1$. By our inductive hypothesis, the claim is true for $k-1$ so, by removing a finite number of vertices, we may assume that all edges from $A_{i-1}$ to $A_i$ are blue and all edges from $A_i$ to $A_{i-1}$ are red for every $i\in\{1,\dots,k-1\}$.

Consider $A_{k-1}\cup A_k$ and suppose there is a red edge from $x\in A_{k-1}$ to $y\in A_k$. All edges from $y$ to $A_{k-1}\backslash\{x\}$ are blue or else we would have a red directed path of length $k+1$ from $x$ which contradicts $x\in A_{k-1}$.  

In the original graph, there is a red directed path $P$ of length $k$ from $y$ because $y\in A_k$. Let $P=p_0p_1\dots p_k$ with $p_0=y$. If, for every $t\in[k]$, we have that $p_t\neq x$, then $xyp_1\dots p_k$ is a red directed path of length $k+1$ from $x$, which contradicts $x\in A_{k-1}$. Therefore, $x=p_t$ for some $t\in[k]$. No edges are removed when we apply the inductive hypothesis and so there are at least $k$ vertex-disjoint paths of length $k-1$ from $x$. Therefore, there exists a red directed path $Q$ of length $k-1$ from $x$ which is vertex-disjoint from $yp_1\dots p_{t-1}$. The path $yp_1\dots p_{t-1}\cup Q$ is a red directed path from $y$ of length $t+k-1$. If $t\neq1$, this path has length strictly larger than $k$, which contradicts $y\in A_k$. Therefore we must have that $t=1$ and the edge $(y,x)$ must be red. 

Consider any vertex $u\in A_{k-1}$ with $u\neq x$. We can find a red directed path of length $k-1$ from $x$ that does not contain $u$ or $y$. To prevent a red directed path of length $k+1$ from $u$, the edge $(u,y)$ must be blue. Therefore, all edges from $A_{k-1}\backslash\{x\}$ to $y$ must be blue (see Figure \ref{fig:doublerededge}).

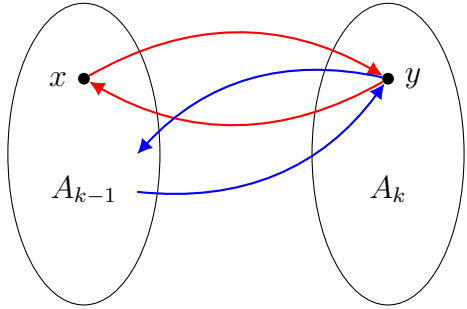
\begin{figure}[H]
\centering
\begin{tikzpicture}
\draw (-2,0) ellipse (1cm and 2cm) node [label={below:$A_{k-1}$}] {};
\draw (2,0) ellipse (1cm and 2cm) node [label={below:$A_k$}] {};
\node[main node, label={180:$x$}] (x) at (-2,1) {};
\node[main node, label={0:$y$}] (y) at (2,1) {};
\draw[-{Latex[length=2mm,width=2mm]},red,thick] (x) to [bend left] (y);
\draw[-{Latex[length=2mm,width=2mm]},red,thick] (y) to [bend left] (x);
\draw[-{Latex[length=2mm,width=2mm]},blue,thick] (y) to [bend right] (-1.3,0);
\draw[-{Latex[length=2mm,width=2mm]},blue,thick] (-1.3,-0.5) to [bend right] (y);
\end{tikzpicture}
\caption{The colours of edges incident to $y$.}
\label{fig:doublerededge}
\end{figure}

Hence, if there is a red edge $xy$ from $A_{k-1}$ to $A_k$, the corresponding backwards edge is also red and all edges between $A_{k-1}\backslash\{x\}$ and $y$ in both directions are blue. Suppose there are an infinite number of red vertex-disjoint edges from $A_{k-1}$ to $A_k$. As both $A_{k-1}$ and $A_k$ contain a blue directed path of density $\frac1r$, we can use these and the pairs of vertices corresponding to the red edges to create a blue directed path of density $\frac2r$. This contradicts our assumption that all blue directed paths have density at most $\frac1r$. Therefore, there are only a finite number of red vertex-disjoint edges from $A_{k-1}$ to $A_k$. By removing a finite number of vertices, we can ensure that all edges from $A_{k-1}$ to $A_k$ are blue.

In order to show that we can remove a finite number of vertices to get all edges from $A_k$ to $A_{k-1}$ to be red, we use exactly the same argument as in the base case.
\\
\\
\noindent
It follows from the inductive argument above that we can remove a finite number of vertices in order to obtain the following structure: for every $i\in\{1,\dots,r-1\}$, all edges from  $A_{i-1}$ to $A_i$ are blue and all edges from $A_i$ to $A_{i-1}$ are red.

We now consider the edges between $A_i$ and $A_j$ where $i<j$ and $i\neq j-1$. All edges from $A_i$ to $A_j$ are blue. If this were not the case, we could find a red directed path of length $j+1$ from a vertex in $A_i$, which is a contradiction. Next, we will show that we can remove a finite number of vertices to get all edges from $A_j$ to $A_i$ being red. If there are no blue edges from $A_j$ to $A_i$, we are done. Otherwise, we may greedily construct a maximal matching $M$ of blue edges from $A_j$ to $A_i$. Using the same argument as the induction base case, we get that $M$ must be a finite matching; if this was not the case, then we would be able to find a blue directed path of density $\frac2r$, which contradicts our original assumption. By deleting the vertices which are incident to an edge in $M$, we will remove all of the blue edges from $A_j$ to $A_i$. We only deleted a finite number of vertices and now all of the edges from $A_j$ to $A_i$ are red as required.

There are only a finite number of pairs $(A_i,A_j)$ to consider. Therefore, by removing a finite number of vertices, we can ensure that, for every $i<j$, all edges from $A_i$ to $A_j$ are blue and all edges from $A_j$ to $A_i$ are red.

Finally, we look at the edges between vertices in the same set. Suppose, for some $i\in\{0,\dots,r-1\}$ there exist $x,y\in A_i$ such that the edge from $x$ to $y$ is red. Then we can find a red directed path of length $i+1$ from $x$, which contradicts $x\in A_i$. Hence, all edges between vertices in the same set $A_i$ are blue.
\end{proof}

\section{Restricting Path Length for Many Colours}
\label{restrictedmanycolours}
\begin{proof}[Proof of Theorem \ref{thm:restrictedmanycolours}]
Let $k\geq2$. Fix a $(k+1)$-colouring of $\vec{K}_{\mathbb{N}}$ where the colours are $c_1,\dots,c_{k+1}$. Assume that, for every $t\in[k]$, there is no $c_t$-coloured directed path of length $r_t$.
\\
\\
\textit{Inductive Hypothesis:} For fixed $i\in[k]$, there exists a set $A\subseteq\mathbb{N}$ such that, for every $t\in[i]$, the $c_t$-indegree of vertices in $A$ is at most $r_t$. Furthermore, $\bar{d}(A)\geq\prod_{t=1}^i\frac1{r_t}$.
\\
\\
\textit{Base Case:} Partition the vertices into sets $A_0,\dots,A_{r_1-1}$ where $n\in A_j$ if and only if the longest directed path of colour $c_1$ from $n$ has length $j$. Fix $j$ such that $\bar{d}(A_j)\geq\frac1{r_1}$. By Claim 1 in the proof of Theorem \ref{thm:restrictedred}, we find that the $c_1$-indegree of vertices in $A_j$ is at most $j\leq r_1$.
\\
\\
\textit{Inductive Step:} Let $i>1$. By our inductive hypothesis, there exists a set $A\subseteq\mathbb{N}$ with $\bar{d}(A)\geq\prod_{t=1}^{i-1}\frac1{r_t}$ such that, for every $t\in[i-1]$, the $c_t$-indegree of vertices in $A$ is at most $r_t$. Partition $A$ into $r_i$ sets $A_0,\dots,A_{r_i-1}$ where $n\in A_j$ if and only if the longest directed path of colour $c_i$ from $n$ in $A$ has length $j$. Fix $j$ such that $\bar{d}(A_j)\geq\frac1{r_i}\bar{d}(A)$. Arguing again as in Claim 1 in the proof of Theorem \ref{thm:restrictedred}, we find that the $c_i$-indegree of vertices in $A_j$ is at most $j\leq r_i$.
\\
\\
Let $A$ be the set obtained by applying the inductive hypothesis with $i=k$. Let $B=\{v\in A:v\text{ has finite }c_{k+1}\text{-outdegree in }A\}$. Using a similar argument to Claim 2 in the proof of Theorem \ref{thm:restrictedred}, we find that $|B|\leq\sum_{t=1}^kr_t$.
	
As in the proof of Theorem \ref{thm:restrictedred}, we may assume, without loss of generality, that $B=\emptyset$. We list the vertices of $A$ in order $v_1,v_2,\dots$ and create a directed path $P=\{p_j\}$ of colour $c_{k+1}$ as follows:
\begin{itemize}
\item Set $p_1=v_1$.
\item Given $p_j$, let $s=\min\{s':v_{s'}\notin P\}$. If the edge $(p_j,v_s)$ is colour $c_{k+1}$, set $p_{j+1}=v_s$. If not, there exists $u\notin P$ such that the edges $(p_j,u)$ and $(u,v_s)$ are both colour $c_{k+1}$ because $p_j$ has an infinite number of out-neighbours with colour $c_{k+1}$ and $v_s$ has at most $\sum_{t=1}^kr_t$ in-neighbours which are not colour $c_{k+1}$. Set $p_{j+1}=u$ and $p_{j+2}=v_s$.
\end{itemize}
The directed path $P$ of colour $c_{k+1}$ contains every vertex of $A$ and hence has upper density $\bar{d}(A)$.
\end{proof}
\noindent
This result is the best possible as shown by the following construction. Partition the vertices of $\vec{K}_{\mathbb{N}}$ into $\prod_{t=1}^kr_t$ sets based on their residue modulo $r_t$ for each $t\in[k]$. The vertex $n\in\mathbb{N}$ is in the set $A_{i_1,i_2,\dots,i_k}$ if and only if, for each $t\in[k]$, $n\equiv i_t\mod r_t$.

Consider vertices $m,n\in\mathbb{N}$. If $m,n\in A_{i_1,\dots i_k}$, then colour both $(m,n)$ and $(n,m)$ colour $c_{k+1}$. If not, suppose that $m\in A_{i_1,\dots,i_k}$ and $n\in A_{j_1,\dots,j_k}$. Let $t=\min\{t'\in[k]:i_{t'}\neq j_{t'}\}$. If $i_t<j_t$, then colour $(m,n)$ colour $c_{k+1}$ and $(n,m)$ colour $c_t$ (see Figure \ref{fig:extremalredgreen}).

\begin{figure}[H]
\centering
\begin{tikzpicture}
\draw (-4,2) ellipse (1.5cm and 1cm) node {$A_{0,0}$};
\draw (0,2) ellipse (1.5cm and 1cm) node {$A_{0,1}$};
\draw (4,2) ellipse (1.5cm and 1cm) node {$A_{0,2}$};
\draw (-4,-2) ellipse (1.5cm and 1cm) node {$A_{1,0}$};
\draw (0,-2) ellipse (1.5cm and 1cm) node {$A_{1,1}$};
\draw (4,-2) ellipse (1.5cm and 1cm) node {$A_{1,2}$};

\draw[-{Latex[length=2mm,width=2mm]},green,thick] (-1,2.2) to (-3,2.2);
\draw[-{Latex[length=2mm,width=2mm]},green,thick] (3,2.2) to (1,2.2);
\draw[-{Latex[length=2mm,width=2mm]},blue,thick] (-3,1.8) to (-1,1.8);
\draw[-{Latex[length=2mm,width=2mm]},blue,thick] (1,1.8) to (3,1.8);
\draw[-{Latex[length=2mm,width=2mm]},green,thick] (3.3,2.6) to [bend right] (-3.3,2.6);
\draw[-{Latex[length=2mm,width=2mm]},blue,thick] (-3,2.4) to [bend left] (3,2.4);

\draw[-{Latex[length=2mm,width=2mm]},green,thick] (-1,-1.8) to (-3,-1.8);
\draw[-{Latex[length=2mm,width=2mm]},green,thick] (3,-1.8) to (1,-1.8);
\draw[-{Latex[length=2mm,width=2mm]},blue,thick] (-3,-2.2) to (-1,-2.2);
\draw[-{Latex[length=2mm,width=2mm]},blue,thick] (1,-2.2) to (3,-2.2);
\draw[-{Latex[length=2mm,width=2mm]},green,thick] (3.3,-2.6) to [bend left] (-3.3,-2.6);
\draw[-{Latex[length=2mm,width=2mm]},blue,thick] (-3,-2.4) to [bend right] (3,-2.4);

\draw[-{Latex[length=2mm,width=2mm]},red,thick] (-3.8,-1.5) to (-3.8,1.5);
\draw[-{Latex[length=2mm,width=2mm]},blue,thick] (-4.2,1.5) to (-4.2,-1.5);
\draw[-{Latex[length=2mm,width=2mm]},red,thick] (0.2,-1.5) to (0.2,1.5);
\draw[-{Latex[length=2mm,width=2mm]},blue,thick] (-0.2,1.5) to (-0.2,-1.5);
\draw[-{Latex[length=2mm,width=2mm]},red,thick] (4.2,-1.5) to (4.2,1.5);
\draw[-{Latex[length=2mm,width=2mm]},blue,thick] (3.8,1.5) to (3.8,-1.5);

\draw[-{Latex[length=2mm,width=2mm]},red,thick] (3.5,-1.2) to (0.9,1.4);
\draw[-{Latex[length=2mm,width=2mm]},blue,thick] (0.7,1.2) to (3.3,-1.4);
\draw[-{Latex[length=2mm,width=2mm]},red,thick] (-3.5,-1.2) to (-0.9,1.4);
\draw[-{Latex[length=2mm,width=2mm]},blue,thick] (-0.7,1.2) to (-3.3,-1.4);
\draw[-{Latex[length=2mm,width=2mm]},red,thick] (0.7,-1.2) to (3.3,1.4);
\draw[-{Latex[length=2mm,width=2mm]},blue,thick] (3.5,1.2) to (0.9,-1.4);
\draw[-{Latex[length=2mm,width=2mm]},red,thick] (-0.7,-1.2) to (-3.3,1.4);
\draw[-{Latex[length=2mm,width=2mm]},blue,thick] (-3.5,1.2) to (-0.9,-1.4);

\draw[-{Latex[length=2mm,width=2mm]},red,thick] (3,-1.4) to (-3,1.6);
\draw[-{Latex[length=2mm,width=2mm]},blue,thick] (-3.1,1.4) to (2.9,-1.6);
\draw[-{Latex[length=2mm,width=2mm]},red,thick] (-3,-1.4) to (3,1.6);
\draw[-{Latex[length=2mm,width=2mm]},blue,thick] (3.1,1.4) to (-2.9,-1.6);

\draw[-{Latex[length=2mm,width=2mm]},blue,thick] (-4.6,1.7) to (-4.6,2.3);
\draw[-{Latex[length=2mm,width=2mm]},blue,thick] (-4.9,2.3) to (-4.9,1.7);
\draw[-{Latex[length=2mm,width=2mm]},blue,thick] (-4.6,-1.7) to (-4.6,-2.3);
\draw[-{Latex[length=2mm,width=2mm]},blue,thick] (-4.9,-2.3) to (-4.9,-1.7);

\draw[-{Latex[length=2mm,width=2mm]},blue,thick] (4.6,1.7) to (4.6,2.3);
\draw[-{Latex[length=2mm,width=2mm]},blue,thick] (4.9,2.3) to (4.9,1.7);
\draw[-{Latex[length=2mm,width=2mm]},blue,thick] (4.6,-1.7) to (4.6,-2.3);
\draw[-{Latex[length=2mm,width=2mm]},blue,thick] (4.9,-2.3) to (4.9,-1.7);

\draw[-{Latex[length=2mm,width=2mm]},blue,thick] (-0.6,1.7) to (-0.6,2.3);
\draw[-{Latex[length=2mm,width=2mm]},blue,thick] (0.6,2.3) to (0.6,1.7);
\draw[-{Latex[length=2mm,width=2mm]},blue,thick] (-0.6,-1.7) to (-0.6,-2.3);
\draw[-{Latex[length=2mm,width=2mm]},blue,thick] (0.6,-2.3) to (0.6,-1.7);
\end{tikzpicture}
\caption{Extremal colouring when $r_1=2$ and $r_2=3$.}
\label{fig:extremalredgreen}
\end{figure}
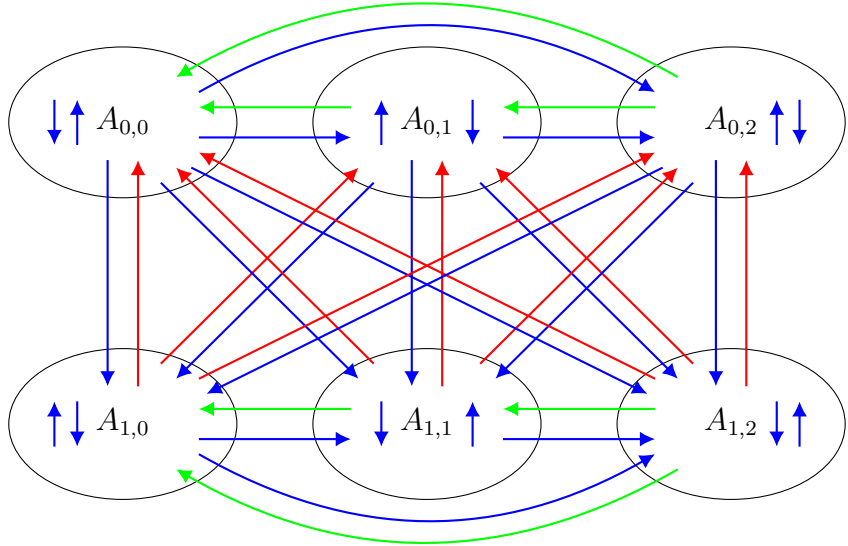
\noindent
It is straightforward to see that, for every $i\in[k]$, the longest directed path of colour $c_t$ has length $r_t-1$. Any directed path of colour $c_{k+1}$ has infinite intersection with at most one of the sets $A_{i_1,\dots,i_k}$ and hence has upper density at most $\prod_{t=1}^k\frac1{r_t}$.

\section{Conclusion}
Consider a 2-colouring of $\vec{K}_{\mathbb{N}}$ that does not contain a red directed path of length $r$ and partition the vertices into $r$ sets, $A_1,\dots,A_{r-1}$, where $v\in A_i$ if and only if the longest directed red path from $v$ has length $i$. It follows from the proof of Theorem \ref{thm:restrictedred} that, if $A_i$ is infinite, there is a blue directed path covering all but a finite number of vertices in $A_i$. Hence, there exist at most $r$ disjoint blue paths covering all but a finite number of vertices of $\vec{K}_{\mathbb{N}}$.

In fact, if $r<4$, it is possible to prove that we may cover all of the vertices of $\vec{K}_{\mathbb{N}}$ with at most $r$ vertex-disjoint blue paths. We therefore make the following conjecture.

\begin{conjecture}
Let $r\geq1$. Take any 2-colouring of $\vec{K}_{\mathbb{N}}$ that does not contain a red directed path of length $r$. Then the vertices of $\vec{K}_{\mathbb{N}}$ can be covered by at most $r$ vertex-disjoint blue directed paths.
\end{conjecture}

In this paper, we have restricted our attention to directed paths. Let $P=v_1v_2...$ be an orientation of a path in $\vec{K}_{\mathbb{N}}$. We say that $v_i$ is a \textit{switch} if either the indegree or the outdegree of $v_i$ in $P$ is 0. The endpoints of $P$ are always switches. If $P$ is a directed path, then the endpoints are the only switches in $P$. DeBiasio and McKenney \cite{DM16} proved that Theorem \ref{thm:densityepsilon} holds for any orientation of a path with finitely many switches. It should be noted that Theorem \ref{thm:density0} still holds for this broader class of paths.

\end{document}